\documentclass[12pt]{article}

\usepackage[a4paper,margin=3cm,innermargin=3cm]{geometry}
\usepackage{amsmath,amsthm,amsfonts,amssymb,latexsym,enumerate,graphicx}
\usepackage{marginnote, etoolbox, imakeidx, needspace}
\usepackage{hyperref}
\usepackage[T1]{fontenc}
\usepackage{comment}
\usepackage{indentfirst}
\usepackage{xcolor}
\usepackage{amsmath}

\DeclareMathOperator*{\argmin}{arg\,min}

\newtheorem{theorem}{Theorem}
\newtheorem{lemma}[theorem]{Lemma}

\newtheorem{proposition}{Proposition}

\newtheorem{remark}[theorem]{Remark}
\newtheorem{observation}[theorem]{Observation}

\def\path{\mathop{\mathrm{\,path\,}}\nolimits}

\date{}
\begin{document}
\title{The continualization approach to the\\ on-line hypergraph coloring}

\author{Akhmejanova Margarita \footnote{Computer, Electrical and Mathematical Sciences and Engineering Division, King Abdullah University of Science and Technology (KAUST), Thuwal 23955-6900, Saudi Arabia.}, Bogdanov Ilya\footnote{Moscow Institute of Physics and Technology, Laboratory of Combinatorial and Geometric Structures, Dolgoprudny, Russia.},  ~Chelnokov Grigory \footnote{National Research University Higher School of Economics, Moscow, Russia.
\newline
margarita.akhmejanova@kaust.edu.sa (First Author), ilya.i.bogdanov@gmail.com (Second Author), grishabenruven@yandex.ru (Third Author).}}

\maketitle

\textbf{Abstract.} The paper deals with an algorithmic problem concerning combinatorial game theory. Here we introduce and analyze a continuous generalization of Chip Game from \cite{Kozik}. The general Chip game was introduced by Aslam and Dhagat \cite{AslDhag} to model on-line type problems on hypergraph coloring.


\section{Introduction}

During the last two decades the algorithmic aspect of classical coloring problems have received new attention. This is in particular due to applied tasks, in which  data is getting piece-by-piece in a serial fashion and respond is required without delay. Another motivation is dealing with the case when  data is too large to hold in memory.
%Online algorithms are well suited for streaming data or when data is too large to hold in memory. Another motivation is  
To simulate this kind of real-world tasks were introduced a new variation of classical coloring problems, called on-line coloring. For instance, $2$-colorability of hypergraphs, also known as ``property B'', was formulated  by  Aslam and Dhagat \cite{AslDhag} in on-line settings. 

\subsection{Classical problems in hypergraph coloring}\label{clas_probl}
Let us briefly recall classical problems on hypergraph coloring. We will use term ``$k$-graph'' for  an $k$-uniform hypergraph, but in general settings we always mean a non-uniform hypergraph.
\begin{itemize}
	\item \textbf{Property B:}
	a hypergraph $H=(V,E)$ has property $B$ (or $2$-colorable) if there is a coloring of $V$ by $2$ colors such that no edge $f \in E$ is monochromatic. Erd\H{o}s and Hajnal \cite{Erd2} (1961) proposed  to find the value $m(k)$ equal to the minimum possible number of edges in a $k$-graph without property $B$. Erd\H{o}s \cite{Erd} (1963--1964) found  bounds $\Omega\left(2^{k}\right) \leq m(k)=O\left(2^{k} k^{2}\right)$ and Radhakrishnan and Srinivasan \cite{RadhSrin} $(2000)$ proved  $m(k)\geq\Omega\left(2^{k}(k / \ln k)^{1 / 2}\right)$. We mention Beck \cite{Beck}(1978) and Duraj,  Gutowski and Kozik \cite{DurGutKozik} (2018), who get a significant progress on $2$-colorability in the non-uniform case. In the most general settings for a given hypergraph  $H=(v,E_1,E_2)$, where $E_1$ and $E_2$ are some sets of edges, we can ask for a black-white coloring such that there are neither black edges in $E_1$ nor white edges in $E_2$.
	
  \item \textbf{Proper coloring:} a hypergraph $H=(V, E)$ is $r$-colorable  if there is a coloring of $V$ by $r$ colors such that no edge $f \in E$ is monochromatic. The general non-uniform case is little known, most results are devoted to the uniform case, see serveys \cite{RaigShab}, \cite{RaigCherk}. 
  For the uniform case, 
 let $m(k,r)$ be the smallest number of edges in a non-$r$-colorable $k$-graph. Bound $m(k,r)=O(k^2r^k\ln r)$ is due to  Erd\H{o}s and bound $m(k,r)\geq\Omega\left(r^{k-1}(k / \ln k)^{r-1 / r}\right)$ is done by Cherkashin and Kozik \cite{CherkKozik} (2014). Also note that when $r>k$ there are stronger results of Akolzin and Shabanov \cite{AkolzinShabanov} (2016).
	
	\item \textbf{Pancromatic coloring:} a vertex $r$-coloring of $H$ is panchromatic if every edge  meets every color.
	%By analogy with $m(k,r)$,  for this problem
	Kostocka \cite{Kost} (2002) determined the number $p(k,r)$ as a mimimum possible number of edges in $k$-graph, which does not admit a panchromatic $r$-coloring and found intresting  relationship between $p(k,r)$ and some other classical characteristics. Cherkashin \cite{Cherk} (2018) proved 
	%  $\Omega\left(\max \left(\frac{n^{1 / 4}}{r \sqrt{r}}, \frac{1}{\sqrt{n}}\right)\left(\frac{r}{r-1}\right)^{n}\right)\leq
	 $p(k, r)=O(k^{2} \ln r\left(\frac{r}{r-1}\right)^{k}/r)$. First author and Balogh recently proved that for all $r^3<k/100 \ln k$ holds that $p(k,r)\geq\Omega((k/\ln k)^{\frac {r-1}{r}}\left(\frac{r}{r-1}\right)^{k}/r^2)$.
	
	\item \textbf{List coloring of $\mathbf{K_{m,m}}$:}
the list chromatic number $\chi_{\ell}(G)$ of a graph $G=(V, E)$ is the minimum integer $r$ such that, for every assignment of a list of $r$ colors to each vertex $V$ of $G$, there is a proper vertex coloring of $G$ in which the color of each vertex is in its list. The study of list colorings  was initiated by Vizing \cite{vizing}(1976) and by Erd\H{o}s, Rubin and Taylor \cite{ErdRubTey}(1980). In particular, Erd\H{o}s, Rubin and Taylor \cite{ErdRubTey} (1980) proved  $
	\chi_{l}\left(K_{m, m}\right)=(1+o(1)) \log _{2}(m) \text { as } m \rightarrow \infty.
	$
	\end{itemize}
 
\subsection{New twist: 
on-line type of 
%classical problems in hypergraph coloring}
classical coloring problems}
Many existing coloring algorithms  become futile on massive graphs, due to their high space and time
complexity. One method for dealing with this is on-line algorithms. Here we describe on-line counterparts to the problems of Section \ref{clas_probl} and give the formal definitions. These counterparts will be single-player games with one player --- called Painter.
\begin{itemize}
	\item \textbf{On-line property B:}
given a hypergraph $H=(V,E)$. Painter does not know  the hypergraph $H$, but he knows the set of edge cardinalities  , i.e. multi-set $\mathcal{A}(H)=\{|e|:e\in E(H)\}$. Let vertex set $V$ be enumerated by $\mathbb{N}$. In round $i$, Painter gets the information about the subset  of edges which
contain vertex $v_i$. Painter must immediately assign color black or
white to the presented vertex $v_i$. Painter wins when all vertices have been colored  and no edge is monochromatic. Denote this game by $(\mathcal{A},2)_{ol}$ game.
For which multisets of edge
cardinalities $\mathcal{A}$ Painter has a winning strategy in $(\mathcal{A},2)_{ol}$ game? 
\\
\\
On-line property B was introduced by Aslam and Dhagat \cite{AslDhag} (1993). They considered the uniform case (when the cardinality of edges are equal, hence $\mathcal{A}=\{k,\ldots, k\}$). Let $m_{o l}(k)$ be the maximal number such that Painter has a winning strategy  for all $k$-graphs with $E(H)\leq m_{o l}(k)$. Aslam and Dhagat \cite{AslDhag} (1993) prowed that $2^{k-1}\leq m_{o l}(k,2)\leq k\cdot\phi^{2k}=O(k(2.62)^k).$ Duraj, Gutowski and Kozik \cite{Kozik} (2015) introduced a new intresting Painter's strategy and improved the above estimate: $m_{o l}(k,2) \leq 16\cdot 2^{k}.$ 
	
	\item \textbf{On-line proper coloring:} by analogy to on-line property $B$ was considered $(\mathcal{A},r)$ game, where Painter just uses $r$ colors. The uniform case, i.e. $m_{ol}(k,r)$ number, have been studied by Khuzieva, Shabanov and Svyatokum \cite{KhuzShavSv}. Repeating arguments from \cite{AslDhag} for the case of $r$ colors they proved 
$r^{k-1}\leq m_{ol}(k,r)\leq k(r-1)^2r^k.$ Unfortunately, the proof of $m_{ol}(k)=\Theta(2^k)$ from \cite{Kozik} can't easily be generalized to the case $r>2$, so the current gap between lower and upper bounds is linear in $k$.

	\item \textbf{On-line panchromatic coloring:} it is a generalisation of $(\mathcal{A},2)$ game for the case of panchromatic coloring. Painter uses $r$ colors and he wins if the final coloring is panchromatic. About the uniform case, i.e. about $p_{ol}(k,r)$ number, is known $r^{-1}\left(r/r-1\right)^{k}\leq p_{ol}(k,r)\leq 3 r(r-1)^{2} k\left(r/r-1\right)^{k+1}$ \cite{KhuzShavSv}.
	
	\item \textbf{On-line list coloring of $\mathbf{K_{m,m}}$:} on-line list coloring was introduced by Schauz \cite{Schauz1},\cite{Schauz2} and Zhu \cite{Zu}. Given a finite bipartite graph $K_{m,m}$. On each round
a set of vertices having a particular color in their lists is revealed and the Painter chooses an independent subset to receive that color. Every vertex is presented as many times as the size of its list. A graph $K_{m,m}$ is said to be $r$-paintable if Painter can produce a list $r$-coloring of $K_{m,m}$ under above conditions. %with no vertex being shown more than $r$ times. 
In the uniform case, i.e. when all list have the same size, was established that $K_{m, m}$ is $(\log _{2} m+O(1))$-paintable as $m \rightarrow \infty$ \cite{Kozik}.
	
%	In round $i$, Painter gets a non-empty set of vertices $v_{i} \subset v \backslash \bigcup_{j=1}^{i-1} x_{j}$. Painter must immediately choose $x_{i}$ that is both a subset of $v_{i}$ and an independent set in $G$. After $i$ rounds, vertices in $\bigcup_{j=1}^{i} x_{j}$ are colored. If a vertex $v$ belongs to exactly $l$ of the sets $v_{1}, \ldots, v_{i}$ we say that $v$ has $l$ permissible colors after $i$ rounds. The game ends when either there exists an uncolored vertex $v$ with $r$ permissible colors (Painter's lose) or all vertices are colored (Painter's win).  Graph $G$ is said to be list on-line $r$-colorable) if Painter has a winning strategy in \emph{List on-line game $(G,r)$}. The minimum $r$ such that $G$ is list on-line $r$ is
%	denoted by $\chi_{o l}(G) .$ Duraj, Gutowski and Kozik \cite{Kozik} established  $\chi_{o l}\left(K_{m, m}\right)=\log _{2} m+O(1) \text { as } m \rightarrow \infty.$
\end{itemize}

All above on-line colorings can be considered as the games between Painter and a second player, called Lister, who builds the hypergraph. So, we might as well ask about the winning strategy of the second player. 

\subsection{The efficient approach for the on-line coloring: Chip game}

It turns out
that in fact on-line hypergraph coloring can be solved by  applying a
method that relies on a technique of chip game, first suggested by Spenser \cite{Spenser} and Aslam and Dhagat \cite{AslDhag}. We first describe chip game in a very general way.

\bigskip
\emph{General Chip game ($S,\tau_{i\in \mathcal F}$)} is determined by a finite
set $S$, called the \emph{set of paths}, an element $dead\in S$, called {\em death} path, and a set of mappings $\tau_{i\in\mathcal{F}}\colon S \to S$,
satisfying $\tau_i(dead)=dead$ (once dead stays dead). The cells of the
playing field are enumerated by $(\mathbb{N}\cup \{0\})\times S$, the cells $(0,path~m): m\neq dead$  are called {\em winning cells}. Some cells contain
chips, there can be more then one chip in the cell. In each round Pusher assigns to each chip whether it stands or
runs on the current round. Then Remover, who sees Pusher's
assignment, picks one of the mappings $\tau_i$.  After that each
standing chip keeps its cell, and each running chip changes its cell
according to the rule $(n,path~m) \to (n-1,path~\tau_i(m))$. Then the new round
begins. Pusher wins if he puts a chip into a winning cell.

\bigskip
Below we show how to formulate online hypergraph coloring problems using Chip game approach.
\begin{itemize}
	\item \textbf{On-line property B:} In terms of ($S,\tau_{i\in \mathcal F}$) game,  
$$S=\{0,1,2,dead\}, \quad
\tau_1:\{0,1,2\}\to \{1, 1, dead \}, \quad \tau_2:\{0,1,2\}\to \{2, dead, 2\}. 
$$
Let us build the bijection between on-line property B and the above ($S,\tau_{i\in \mathcal F}$) chip game. Actually, we will show that  Painter=Remover and Presenter=Pusher and throughout the entire $(\mathcal{A},2)_{ol}$ game every edge $e$ with $i, i<|e|$ colorless vertices and $|e|-i$ vertices colored with $j,j\in\{1,2\}$ one-to-one corresponds to a chip in the cell $(i, path~j)$ and every colorless edge $e$ one-to-one corresponds to a chip in the cell $(|e|, path ~0)$. Now formally. Let $\mathcal{A}$ be the initial chip distribution in the $0$ path. Then, if Remover has a winning strategy with this $\mathcal{A}$ in ($S,\tau_{i\in \mathcal F}$) chip game, then Painter can use this strategy
to win on-line coloring $(\mathcal{A},2)_{ol}$ game. Painter
first represents each edge cardinality $a\in \mathcal{A}$ by the chip in the cell $(a,path~ 0)$. When
Presenter reveals a new vertex $v$ and declares to which edges $v$ belongs, Painter imagines that the Pusher has made a move in which all the running chips are edges containing vertex $v$ and only them. He examines what response action Remover should make, and in a sense "makes" Remover's move, i.e. if the winning strategy for Remover is to take $\tau_j, j\in\{1,2\}$, then Painter colors $v$ with $j$. Similarly, Pusher's winning strategy with $\mathcal{A}$ in a chip game  can be used by Presenter to win $(\mathcal{A},2)_{ol}$ game.

\bigskip
The bijections for on-line proper and on-line panchromatic coloring games are proved in a similar way.

	\item \textbf{On-line proper coloring:} 
In terms of ($S,\tau_{i\in \mathcal F}$) game, 
$$ S=\{0,1,\ldots, r ,dead\}, \quad
\tau_i\colon\{0, 1,\ldots, i, \ldots ,r\}\to\{i, dead,\ldots, i, \ldots ,dead\},$$
i.e. $\tau_{i}$ maps $0$ and $i$ to the $i$ and $S\setminus\{0,i\}$ to the $dead$.
	\item \textbf{On-line panchromatic coloring:}
In terms of ($S,\tau_{i\in \mathcal F}$) game, $$ S= \{0,1\}^r, \quad
\tau_i:j\to j\vee (0,\ldots, 0,1,0,\ldots,0),$$
where $(0,\ldots, 0,1,0\ldots,0)$ has $1$ on the $i$-th position, $(1,\ldots,1)$ is dead path and $\vee$ is the logical OR. As in the previous case there exists a bijection between this chip game and on-line panchromatic coloring.
	\item \textbf{On-line list coloring of $\mathbf{K_{m,m}}$:}
	In terms of General ($S,\tau_{i\in \mathcal F}$) chip game, 
$$ S=\{1, 2,dead\}, \quad
\tau_1\colon\{1, 2\}\to\{1, dead\}, \quad \tau_2\colon\{1, 2\}\to\{dead, 2\}.$$
Actually, in \cite{Kozik} is explained that on-line $k$-list coloring of $K_{m,m}$ is equivalent to the following game given below.
\end{itemize}

\emph{``Chip game''} --- A historical remark. In \cite{Kozik} was presented a special version of chip game, where there is no $0$ road, i.e. in terms of ($S,\tau_{i\in \mathcal F}$) game,  
\begin{equation}
S=\{1,2,dead\}, \quad
\tau_1:\{1,2\}\to \{1, dead \}, \quad \tau_2:\{1,2\}\to \{dead, 2\}. 
\end{equation}
 For this game there is no bijection from on-line property B. But clearly, if in on-line property B game one replaces a chip $(i,path~0)$  by two chips $(i,path~1)$ and $(i,path 2)$, then the new arrangement does not get worse for Pusher. Hence, if Remover has a winning strategy  in this chip game, then Painter can use this strategy
to win on-line coloring $(N,2)_{ol}$ game. And if Pusher has a winning strategy in the chip game, then Presenter can use this strategy to win on-line coloring $(2N,2)_{ol}$ game.   For more details, see \cite{Kozik}.
%However, Duraj, Gutowski and Kozik in their work \cite{Kozik} proved the following fact. Let in the starting configuration only cells $(1,path~k)$ and $(2,path~k)$ are occupied, each of
%them by exactly $N$ chips and let $(N,2)_{ol}$ game be a $(\mathcal{A},2)_{ol}$ game, where $\mathcal{A}=\{k,\ldots,k\}$ with $|\mathcal{A}|=N$. Then, if Remover has a winning strategy  in this chip game, then Painter can use this strategy
%to win on-line coloring $(N,2)_{ol}$ game. And if Pusher has a winning strategy in the chip game, then Presenter can use this strategy to win on-line coloring $(2N,2)_{ol}$ game.

\section{The continualization approach to the on-line hypergraph coloring}

Discrete processes can sometimes be analysed by placing them in a continuous framework. Here we consider a generalisation of \emph{General Chip  ($S,\tau_{i\in \mathcal F}$)} game, where instead of chips we have non-negative real numbers, called gold sand. One can consider these numbers as the amount of gold sand left by crashing chips in the chip game). Now formally. 

\bigskip
\emph{General Gold Sand game ($S,\tau_{i\in \mathcal F}$)} is determined by a finite
set $S$, called the \emph{set of paths}, an element $dead\in S$, called {\em death} path, and a set of mappings $\tau_{i\in\mathcal{F}}\colon S \to S$,
satisfying $\tau_i(dead)=dead$ (once dead stays dead). The cells of the
playing field are enumerated by $(\mathbb{N}\cup \{0\})\times S$, the cells $(0,path~m): m\neq dead$  are called {\em winning cells}. All cells contain real non-negative numbers, called \emph{gold sand}. There is only finite amount of non-zero numbers. In each round Pusher splits gold sand in each
cell into two parts, called {\em standing} and {\em running}. Then Remover (knowing how Pusher shared and knowing the current amount of gold sand in each cell) picks one of the mappings $\tau_i$.  After that each standing part keeps its cell, and each running part changes its cell according to the rule $(n,path~m) \to (n-1,path~\tau_i(m))$. 
Moreover, all sand from the dead path is removed from the field and all sand from the winning cells instantly becomes Pusher's win and is also removed from the field. Then the new round begins. 

\bigskip
\emph{The question:}
let $ x=[x_{i,path~j}]_{([N]\cup \{0\})\times S}$ denotes the vector of initial distribution of gold sand, in short, arrangement $ x$. Let $\mathcal{X}$ be the set of all vectors of initial distribution of gold sand. Can one find the supremum of Pusher's win for any given arrangement $ x\in\mathcal{X}$?

\subsection{Our results}
%\textcolor{red}{There is a difficulty for the uniform notation of $g(x,p)$ because of the weight of gold sand on the position $x_{0,path~j}$. }

\begin{proposition}[Continuous on-line property B]
	 For each $p \in[0,1]$ through $ w(p)$ denote the vector $ w_{p} \in \mathcal{X}$, such that $\left(w_{p}\right)_{0, p a t h~0}=$ $\left(w_{p}\right)_{0, \text { path } 1}=\left(w_{p}\right)_{0, \text { path } 2}=1$ and
	 \begin{equation*}
	 	w(p)_{i, p a t h~j}=\left\{\begin{array}{rll}
	 		p^{i}+(1-p)^{i} & \text { if } & j=0 \\
	 		p^{i} & \text { if } & j=1 \\
	 		(1-p)^{i} & \text { if } & j=2
	 	\end{array}\right.
	 \end{equation*}
	for $i \in \mathbb{N}$.
	Then the supremum of Pusher's win in Continuous on-line property B game on the initial arrangement $ x$ is $\min _{p}  x \cdot  w(p)$ \footnote{operation $\cdot$ is scalar(dot) product of two vectors.}.
\end{proposition}

\begin{proposition}[Continuous on-line proper coloring]
	 For each r-plet $p=\left(p_{1}, \ldots, p_{r}\right)$, such that $p_{i} \geqslant 0$ and $p_{1}+\cdots+p_{r}=1$, through $ w(p)$ denote the vector $ w(p) \in \mathcal{X}$, such that $w(p)_{0, \text { path } 0}=\left(w_{p}\right)_{0, \text { path } 1}=\cdots=\left(w_{p}\right)_{0, \text { path }r}=1$ and
	\begin{equation*}
	w(p)_{i, p a t h ~j}=\left\{\begin{array}{rlr}
		p_{1}^{i}+p_{2}^{i}+\cdots p_{r}^{i} & \text { if } \quad & j=0 \\
		p_{j}^{i} & \text { if } & j>0
	\end{array}\right.
	\end{equation*}
	for $i \in \mathbb{N}$.
	Then the supremum of Pusher's win in Continuous on-line proper coloring game on the initial arrangement $x$ is $\min _{p}  x \cdot  w(p)$.
\end{proposition}

\begin{proposition}[Continuous on-line panchromatic coloring]
 For each r-plet $p=\left(p_{1}, \ldots, p_{r}\right)$, such that $p_{i} \geqslant 0$ and $p_{1}+\cdots+p_{r}=1$, through $ w(p)$ denote the vector $ w(p) \in \mathcal{X}$, such that  $w(p)_{0, \text { path } 0}=\left(w_{p}\right)_{0, \text { path } 1}=\cdots=\left(w_{p}\right)_{0, \text { path }2^{r}-1}=1$ and for all $i \in \mathbb{N}$
$$
w(p)_{i, p a t h~j}=\sum_{\mathcal{M}\subseteq \{0,1\}^r:~\mathcal U_j<\mathcal M}\left(S(\mathcal M)\right)^i(-1)^{r-|\mathcal M|-1}$$ 
where $\mathcal{U}_j$ is binary code of path $j$ and $S(\mathcal M)=\sum_{i\in \mathcal M} p_i$, and $<$ is bitwise.
%w(p)_{i, p a t h~j}=\sum_{\mathcal{M}:~\mathcal U_j \subseteq \mathcal M \subset [1,r]}\left(S(\mathcal M)\right)^i(-1)^{r-|\mathcal M|-1}$$ 
%with  $S(\mathcal M)=\sum_{i\in \mathcal M} p_i$
%and $\mathcal{U}_j$ is defined as follows: let $(a_1,a_2,\ldots,a_r)$ is binary code of path $j$ with exactly $s, s\in [1,r]$ units, and let $\mathcal U_j=\{j_1,j_2,\ldots,j_s\}\subset [1,r]$ be a tuple of indexes, where $a_{j_1}=\ldots=a_{j_s}=1.$ 
See example \footnote{In case $r=4$ we have $f_j($1111$)=0, f_j(1110)=(p_1+p_2+p_3)^j,
	f_j(1100)=(p_1+p_2+p_3)^j+(p_1+p_2+p_4)^j-(p_1+p_2)^j, f_j(1000)=(p_1+p_2+p_3)^j+(p_1+p_2+p_4)^j+(p_1+p_3+p_4)^j-(p_1+p_2)^j-(p_1+p_3)^j-(p_1+p_4)^j+p_1^j$. And similarly, for other paths.}.
Then the supremum of Pusher's win in Continuous on-line proper coloring game on the initial arrangement $ x$ is $\min _{p}  x \cdot  w(p)$.
\end{proposition}

\begin{proposition}[Continuous on-line list coloring of $K_{m,m}$]
 For each $p \in[0,1]$ through $ w(p)$ denote the vector $ w_{p} \in \mathcal{X}$, such that $\left(w_{p}\right)_{0, \text { path } 1}=\left(w_{p}\right)_{0, \text { path } 2}=1$ and
\begin{equation*}
	w(p)_{i, p a t h ~j}=\left\{\begin{array}{rll}
		p^{i} & \text { if } & j=1 \\
		(1-p)^{i} & \text { if } & j=2
	\end{array}\right.
\end{equation*}
for $i \in \mathbb{N}$.
Then the supremum of Pusher's win in Continuous on-line list coloring of $K_{m,m}$ game on the initial arrangement $ x$ is $\min _{p}  x \cdot  w(p)$.
\end{proposition}

\section{Proof of Proposition 1}
\subsection{Notation}

We denote
$[N]=\{1,\ldots,N\}$, $\mathcal A=\mathbb R_+^{3N+3}\setminus\{0\}$ and $\mathcal{V}=\mathbb R^{3N+3}$.

While considering any game of the described above type, we denote by $E( x)$  the value of this game with the initial arrangement $ x$ to Pusher in the commonly accepted sense of the value, that is, the maximal sum of the golden sand, reached winning cells, if both players play optimally.
The initial (and thus any) arrangement is regarded further as a vector 
$$ 
x=[x_{i,path~j}]_{([N]\cup \{0\})\times\{0,1,2\}}.
$$
%By $r$ we denote a running vector from On-line property B game.\\
%Let $ r$ is a running vector from On-line property B game. Denote by $r_{1}$ a vector with zero coordinates on paths $2$ and the same coordinates on paths $1$ and zero, by $r_{2}$ a vector with zero coordinates on paths $1$ and the same coordinates on paths $2$ and zero, by $r_0$ a vector with zero coordinates on paths $1$ and $2$ and the same coordinates on path $0$. Hence, 
%
%$$ 
%r= r_{1}+ r_2 - r_0. 
%$$
%Let $\overrightarrow r$ is a running vector from On-line property B game. 
For a given vector $x\in\mathcal{V}$ denote by $x_{\text{path~j}}, j\in\{0,1,2\}, x\in R_+^{N+1},$  %a vector with the same coordinates on path $j$ and zero coordinates on other paths, so that
it's projection on $path~j$. Observe that $w(p)_{path~0}=w(p)_{path~1}+w(p)_{path~2}$ and 
%$$ x=x_{\text{path~0}}+x_{\text{path~1}}+ x_{\text{path~2}}.$$
\begin{gather}\label{new_weight}
x\cdot w(p) = (x_{path~1}+x_{path~0})\cdot w(p)_{path~1}+(x_{path~2}+x_{path~0})\cdot w(p)_{path~2}.
\end{gather}
For given mapping $\tau$ and vector $x\in\mathcal{A}$ determine a operation of shifting, $\overrightarrow{x}(\tau)$, that transforms each coordinate  
$x_{n,path~m},n\geq 1$ to the $x_{n-1,\tau(path~m)}$.  We will sometimes omit $\tau$ in cases when it is obvious from the context, which permutation we mean.

By the \emph{norm} of a vector $ v=[v_{i,path~j}]_{([N]\cup \{0\})\times\{0,1,2\}},  v\in\mathcal{V}$ we mean the $\ell_1$ norm, that is the sum of all it's coordinate modules:
$$\| v\|_1=\sum_{ ([N]\cup\{0\})\times \{0,1,2\}} |v_{i, path~j}|.$$

By $\rho( v,M)$ we denote the standard  \emph{distance} from vertex $ v$ to~the set $\mathcal M\subseteq\mathcal{V}$,
$\rho( v,M)=\inf_{m\in \mathcal M}\|{v-m}\|_1.$
We also introduce seminorm
%s $\| v\|_{path~j}, j\in\{0,1,2\}$ equal the sum of all coordinate modules on path $j$ and  $\| v\|_2$ equals the sum of all coordinate modules $|v_{i,path~j}|$ with indexes $i\geq 2$ over all paths: 
$\| v\|_{1,2}$ equals the sum of all coordinate modules $|v_{i,path~j}|$ with indexes $i\geq 2$ over all paths: 
\begin{gather*}
%\| v\|_{path~0}=\sum_{i\in [N]}|v_{i, path~0}|, \quad \| v\|_{path~1}=\sum_{i\in [N]}|v_{i, path~1}|, \quad \| v\|_{path~2}=\sum_{i\in [N]}|v_{i, path~2}|, \\
\| v\|_{1,2}=\sum_{(i,j)\in ([N]\setminus\{1\})\times\{0,1,2\}}|v_{i, path~j}|.
\end{gather*}

For any $ x\in \mathcal A$, we define the infimum of the scalar product $x$ and $w(p)$~ in $p\in[0,1]$:
\begin{gather*}
%g( x, p)= x\cdot  w(p),\\
%e( x)=\inf_{p\in[0,1]}g( x,p).
e( x)=\inf_{p\in[0,1]}  x\cdot  w(p).
\end{gather*}

%\bigskip
Theorem~1 claims that $e( x)=E( x)$.

\begin{remark}
Since the weight vector $ w(p)$ has all coordinates no more than one, the scalar product of $ x$ with $ w(p)$ is at most $\| x\|_1$. Hence, $e( x)\leq\| x\|_1$. Moreover, $e( x)=0$ if and only if either $x_{path~1}$ or $x_{path~2}$ vanishes. 
\end{remark}

Finally, denote
\begin{equation}\label{p*}
	p^*_x=\argmin_{p\in[0,1]} x\cdot w(p).
\end{equation} 
%$$
%	(x_{path~1}+x_{path~0})\cdot w'(p)_{path~1}+(x_{path~2}+x_{path~0})\cdot w'(p)_{path~2}.
%$$
%	\sum_{i\in [N]}\left[(p^{i-1}+(1-p)^{i-1}) ix_{i,path~0}+ip^{i-1} x_{i,path~1}-i(1-p)^{i-1}x_{i,path~2}\right].

\subsection{Lipschitz property of $e$ and $E$}

\begin{proposition}
  \label{W-Lip}
 For any $ x,  y\in\mathcal A$, we have
  $$
    e( x)-e( y)\leq 
\|{x-y}\|_1 \qquad{and}\qquad
    E( x)-E( y)\leq 
\|{x-y}\|_1.
  $$
\end{proposition}

Indeed we prove that $E( x)-E(y)$ and $e( x)-e( y)$ does not exceed the sum of non-negative coordinates of the vector ${x-y}$.
\begin{proof}
\begin{gather*}
e( x)-e( y)= x\cdot w(p^*_x)- y\cdot w(p^*_y)=  x\cdot w(p^*_x)- x\cdot w(p^*_y)+{(x-y)}\cdot w(p^*_y)\leq \|{x-y}\|_1,
\end{gather*}
where we used that  $p^*_x$ is argminimum of $ x\cdot  w(p)$, and so $ x\cdot  w(p^*_x)- x\cdot  w(p^*_y)\leq 0$.

%The first part of Proposition \ref{W-Lip} is proved. 
Now prove the second part of Proposition \ref{W-Lip}.
%\begin{equation}\label{lip}
%E( x)-E( y)\leq\sum_{(i,j):x_{i,path~j}>y_{i,path~j}} (x_{i,path~j}-y_{i,path~j})\leq \|{x-y}\|.
%\end{equation}
It is enough to prove when $ y$ differs from $ x$ only in one coordinate. Without loss of generality, let $(i, path~1)$, $i\in[N]$ be the desired coordinate, such as $x_{i,path~1}>y_{i,path~1}$ and other coordinates of vectors $ x$ and $ y$ are equal. Assume for a moment that Remover knows a strategy how to play  with the initial arrangement $ y$, so that Pusher couldn't win more than $E( y)$. Denote this strategy by $R( y)$. Then for the arrangement $ x$ Remover can apply the following strategy: he divides gold sand on the position $i$ of path $1$  into two parts size of $(x_{i,path~1}-y_{i,path~1})$ and $y_{i,path~1}$. Remover applies his strategy $R( y)$ by imaging the first part of gold sand is fake and totally ignoring it. Then the final Pusher'win will be at most $E( y)$ plus the ignored gold $(x_{i,path~1}-y_{i,path~1})$. Proposition \ref{W-Lip} is proved.
\end{proof}

\bigskip
\let\eps\varepsilon
In the next chapter we give an explicit Remover's strategy not allowing Pusher to win more than $e( x)$.

\subsection{Remover's strategy}\label{Remover's strategy}

 Let $r$ be the running part chosen by Pusher in some move. Due to~\eqref{new_weight} we have
\begin{multline}
  p^*_xr\cdot w(p^*_x)+(1-p^*_x)r\cdot w(p^*_x)
  = r\cdot w(p^*_x)\\ 
  = (r_{path~1}+r_{path~0})\cdot w(p^*_x)_{path~1}
    +(r_{path~2}+r_{path~0})\cdot w(p^*_x)_{path~2}\\
  = p^*_x\overrightarrow r(\tau_1)\cdot w(p^*_x)
    + (1-p^*_x)\overrightarrow r(\tau_2)\cdot w(p^*_x).
  \label{balance}
\end{multline}
Therefore, one of the inequalities
\begin{align}
  p^*_xr\cdot w(p^*_x) 
  &\geq  (r_{path~1}+r_{path~0})\cdot w(p^*_x)_{path~1} 
    = p^*_x\overrightarrow r(\tau_1)\cdot w(p^*_x),
  \tag{$\tau_1$}\\
  (1-p^*_x)r\cdot w(p^*_x) 
  &\geq  (r_{path~2}+r_{path~0})\cdot w(p^*_x)_{path~2}
    = (1-p^*_x)\overrightarrow r(\tau_2)\cdot w(p^*_x),
  \tag{$\tau_2$}
\end{align}
should hold. In either case, Remover applies the strategy indicated in the brackets at the right. (If $p^*_x=0$, then Remover applies $\tau_2$ as the second inequality also holds; the case $p^*_x=1$ is treated similarly.)

This way, we get
\begin{observation}\label{obs}
%Let $r$ is a running vector from On-line property B game. Then
Let $\tau$ be the Remover's choice indicated above. Then
$$r\cdot w(p^*_x)\geq \overrightarrow{r}(\tau)\cdot w(p^*_x).$$
\end{observation}

%%%%%%%%%%%%%%%%% Прежний текст
\begin{comment}
Recall that 
$$
r\cdot w(p) = (r_{path~1}+r_{path~0})\cdot w(p)_{path~1}+(r_{path~2}+r_{path~0})\cdot w(p)_{path~2}.
$$

Then 
\begin{gather*}
\text{either}\qquad (r_{path~1}+r_{path~0})\cdot w(p^*_x)_{path~1}\geq p^*_xr\cdot w(p^*_x)~~~~~~~~~~~~~~~~~(1)
\end{gather*}
\begin{gather*}
\qquad \text{or}\qquad
(r_{path~2}+r_{path~0})\cdot w(p^*_x)_{path~2}\geq (1-p^*_x)r\cdot w(p^*_x).~~~~~~~~(2)
\end{gather*}

\bigskip
In the first case Remover deletes running part from $path~1$, and otherwise, from $path~2$. 
%\bigskip
\newline
\newline
So, for $\tau$ chosen as above, the following inequality holds:

\begin{observation}\label{obs}
%Let $r$ is a running vector from On-line property B game. Then
$$r\cdot w(p^*_x)\geq \overrightarrow{r}(\tau)\cdot w(p^*_x).$$
\end{observation}
\end{comment}
%%%%%%%%%%%%%%%%%%% Конец прежнего текста

\bigskip
\noindent
Next proposition says that $e(x)=\inf_p x\cdot w(p)$ never increases during the Remover's algorithm.

\bigskip
\begin{proposition}\label{'-0}
	Assume that an arrangement $ y$ obtained from $ x$ by a Pusher's move and Remover's response according to his strategy. Then $ e( x)\geq e( y).$
\end{proposition}

\begin{proof}
	Since
	$$y=\overrightarrow{r}(\tau)+(x-r),$$
	we have
	$$e(x)=x\cdot w(p^*_x)=r\cdot w(p^*_x)+(x-r)\cdot w(p^*_x)\geq$$
using Observation \ref{obs}
$$\overrightarrow{r}\cdot w(p^*_x)+(x-r)\cdot w(p^*_x)=y\cdot w(p^*_x)\geq y\cdot w(p^*_y)=e(y).$$
\end{proof}

Let us now explain why above Remover's strategy not allowing Pusher to win more than $e(x)$. Clearly, when the game ends, the weight of final configuration is not less than Pusher's win, the value $E(x)$. On the other hand, by Proposition \ref{'-0}, $e(x)$ never increases. Hence, Remover cannot win more than $e(x)$. 

\subsection{Idea of the proof of inequality $E\geq e$}

We have already shown that Remover has a strategy not allowing Pusher to win more than $e(x)$. Now we concentrate on proving the converse. For this purpose, we will show that for any $\eps\in(0,1)$, the inequality
$$
E(x)\geq e(x)-2\eps N\|x\|_1
\eqno(*)
$$
holds. 
The main idea is to show that in most of the arrangements, Pusher can perform a move such that $e(x)$ changes much slower than $\|x\|$ --- a precise formulation is given in Lemma~\ref{move} below. We start with the discussion of exceptional arrangements, i.e., those on which the strategy provided in Lemma~\ref{move} does not work.
%Before we present a Pusher's strategy, we introduce some notations and prove several preliminaries lemmas.

%\bigskip
%Now let's prove the second part of Proposition \ref{W-Lip}. It suffices to introduce a Remover's strategy when Pusher win with the initial arrangement $x$ is at most $E(y)+\sum_{i\colon x_i>y_i}(x_i-y_i)$.

\subsection{Degenerate arrangements}

%Let us determine a function $h(x,p)$.

Before presenting degenerate arrangements, let us determine a function $h(x,p)$.
\begin{equation}\label{h}
	h( x,p)=\frac\partial{\partial p} w(p)\cdot x=  w'(p)\cdot x.
\end{equation}

\begin{proposition}	\label{convex}
For any $ x\in\mathcal A$ with $\| x\|_{1,2}>0$, the function $h(x,p)$ is strictly increasing in~$p$. Moreover, the value $p^*_x$ is determined by the equation $h( x,p^*_x)=0$, unless $p^*_x\in\{0,1\}$.
\end{proposition} 

\begin{proof}
For $i>1$ each of the functions $p\mapsto p^i$ and $p\mapsto (1-p)^i$ is strictly convex. Assuming that $\| x\|_2>0$, we get that the function $ x\cdot  w(p)$ is the sum of convex functions, one of which is strictly convex. Thus $x\cdot  w(p)$ is strictly convex and so, attains a unique local minimum (with respect to~$p$) on $[0,1]$, which is determined by the equation $\frac\partial{\partial p}  x\cdot  w(p(x))|_{p=p^*_x}=h( x,p^*_x)=0$ unless $p^*_x\in\{0,1\}$. 	 
\end{proof}

\bigskip
We say that an arrangement $x\in\mathcal A$ is \emph{degenerate} if either $\|x\|_{1,2}=0$, or $p^*_x\in\{0,1\}$. The latter condition, for $\|x\|_{1,2}>0$,  holds if either $h(x,0)\geq 0$ or $h(x,1)\leq 0$, i.e.,

\begin{align*}\label{degen-def}
& \text{either} \qquad \sum_{i>0}i(x_{i,path~2}+x_{i,path~0})\leq x_{1,path~1}+x_{1,path~0} \\
&\qquad \text{or}\qquad  \sum_{i>0}i(x_{i,path~1}+x_{i,path~0})\leq x_{1,path~2}+x_{1,path~0}.
\end{align*}

Notice that in case $h(x,0)\geq 0$ values $x_{i,path~1}$ for $i>1$  does not affect whether the arrangement is degenerate or not. Similarly, in case  $h(x,1)\leq 0$.
 
\bigskip
 We say $x$ is \emph{negatively (resp. positively) degenerate} if $x_{1,path~1}+x_{1,path~0}\geq \sum_{i>0}i(x_{i,path~2}+x_{i,path~0})$
 % and $x_{i}=0$ for $i>1$ 
 (resp. $x_{1,path~2}+x_{1,path~0}\geq \sum_{i>0}i(x_{i,path~1}+x_{i,path~0})$).
 %( and $x_{-j}=0$ for $j>1$). 
 Denote the set of degenerate arrangements by~$\mathcal D$ and set of positively and negatively arrangement by $\mathcal D_{+}$ and $\mathcal D_{-}$, respectively. Notice that $\mathcal D$, $\mathcal D_{+}$ and $\mathcal D_{-}$ are closed. 

Say that an arrangement $x\in\mathcal A$ is \emph{regular} if it is not degenerate. Denote by $\mathcal R=\mathcal A\setminus\mathcal D$ the set of regular arrangements, and notice that $\mathcal R$ is open and convex (as $\mathcal R$ is determined by a system of linear inequalities).

\begin{proposition}
	\label{degen-win}
	For any degenerate arrangement $x\in\mathcal D$ we have $E(x)=e(x)$.
\end{proposition}

\begin{proof}
 We may assume that $x$ is positively degenerate. Then $e(x)=\sum_{i\geq 0} (x_{i,path~1}+x_{i,path~0})$. So, it remains to prove that $E(x)=\sum_{i\geq 0} (x_{i,path~1}+x_{i,path~0})$. We prove the inequality by induction on the maximal index $i$, such that $x_{i,path~1}+x_{i,path~0}\neq 0$. It is trivial for $i = 1$. Let us determine two functions:
\begin{equation*}
	d(x)=\sum_{i\geq 0} i(x_{i,path~1}+x_{i,path~0}) \qquad \text{and} \qquad s(x)=\sum_{i\geq 0} (x_{i,path~1}+x_{i,path~0}).
\end{equation*}
 Then, being positively degenerate is equivalent to $x_{1,path~2}\geq d(x)-x_{1,path~0}$.  Pusher applies the following move: he takes all gold sand from $path~0$ and $path~1$ and $s(x)-x_{1,path~0}$ of gold sand from $(1,path~2)$ as running parts. Then, either Pusher's win on just this step is equal to $\sum_{i\geq 0} (x_{i,path~1}+x_{i,path~0})$ and Proposition \ref{degen-win} is proved, or 
Remover deletes running parts from $path~2$ and Pusher wins $x_{1,path~1}+x_{1,path~0}$. In the second case, we get a new arrangement $x'$, such as
$$d(x')=d(x)-s(x),$$
$$s(x')=s(x)-(x_{1,path~1}+x_{1,path~0}),$$
$$E(x)\geq E(x')-(x_{1,path~1}+x_{1,path~0}).$$
Note that $x'$ is also positively degenerate. Indeed, 
$$x'_{1,path~2}=x_{1,path~2}-(s(x)-x_{1,path~0})\geq d(x)-s(x)=d(x')=d(x')-x'_{1,path~0},$$
where at the end we used that due to Pusher's move $x'_{path~0}=0$. 

So we apply induction hypothesis to $x'$. 
In fact, we prove that $E(\cdot)\geq s(\cdot)$ and we control $d(\cdot)$ just to guarantee that a given arrangement is positively degenerate.  Finally note that we have equality in $E(x)\geq e(x)=s(x)$, since in Chapter \ref{Remover's strategy} we already proved $E(x)\leq e(x)$.
\end{proof}

Next, take any $\eps\in(0,1)$. Say that and arrangement $x\in\mathcal A$ is \emph{$\eps$-degenerate} if $\rho(x,\mathcal D)\leq\eps\|x\|_1$. The set $\mathcal D_\eps$ of $\eps$-degenerate arrangements is also closed, and its complement $\mathcal R_\eps=\mathcal A\setminus\mathcal D_\eps$ is open and convex. The arrangements in $\mathcal R_\eps$ are referred to as \emph{$\eps$-regular arrangements}.

\begin{proposition}
	\label{e-deg}
	For every $\eps>0$, any $\eps$-degenerate arrangement~$x\in\mathcal D_\eps$ satisfies 
	$$E(x)\geq e(x)-2\eps\|x\|_1.$$
\end{proposition}
\begin{proof}
Choose $y \in \mathcal{D}$ with $\|x-y\|_1 \leq \varepsilon\|x\|_1$. By the Lipschitz property of $E$ and $e$ (Proposition \ref{W-Lip}), we have
$$
E(x) \geq E(y)-\|x-y\|_1=e(y)-\|x-y\|_1 \geq e(x)-2\|x-y\|_1 \geq e(x)-2 \varepsilon\|x\|_1,
$$
as desired.
\end{proof}

\subsection{Behaviour of the functions $p^*$, $x\cdot w(p)$, and $e$ on $\eps$-regular vectors}

The analytic properties exhibited in this section are crucial for constructing Pusher's strategy on $\eps$-regular arrangements. 

Without loss of generality, we assume that $p\leq 1-p$.
%$For $p\in[0,1]$, we denote $\langle p\rangle=\min(p,1-p)$.
In what follows, we use a bound
\begin{equation}
\label{h_x}
\left|\frac{\partial h}{\partial x_{i,path~j}}(x,p)\right|\leq i\leq N.
\end{equation}
which is a trivial consequence of~\eqref{h}.

\begin{proposition}
	%  Take any $\eps\in(0,1)$.
	\label{reg-good}
	For any $x\in\mathcal R_\eps$, $\|x\|_{1,2}$ and each of $\|x_{path~j}+x_{path~0}\|_1,j\in\{1,2\}$ are at least $\eps\|x\|_1$.
\end{proposition}

\begin{proof}
%If $\|x\|_+\leq \eps\|x\|_1$, modify $x$ by increasing $x_{-1}$ on $\|x\|_+$. So, we get degenerate arrangement $y$ with $\|x-y\|\leq \eps\|x\|_1$, which contradicts $x\in\mathcal R_\eps$.
\begin{comment}
 Suppose $\|x_{path~1}+x_{path~0}\|\leq \eps\|x\|/N$; the case $\|x_{path~2}+x_{path~0}\|\leq \eps\|x\|/N$ is
 similar. Consider vector $y$, such that $y$ differs from $x$ only in coordinate~$(1,path~2)$ as $y_{1,path~2}=x_{1,path~2}+N\|x_{path~1}+x_{path~0}\|$. Then $y$ is degenerate and $\|x-y\|\leq\eps\|x\|$. Consequently, $x$ is $\eps$-degenerate, which contradicts $x\in\mathcal R_\eps$.
 
 Consider the case $\|x\|_2\leq \eps\|x\|/N.$ We can assume that $x_{1,path~1}\leq x_{1,path~2}.$ Consider vector $y$, such that $y$ differs from $x$ only in coordinate~$(1,path~2)$ as $y_{1,path~2}=x_{1,path~2}+N\|x\|_2$. Then $y$ is degenerate and $\|x-y\|\leq\eps\|x\|$. So again we get to a contradiction.
\end{comment}
We modify $x$ by vanishing the coefficients which less than $\eps\|x\|_1$. By this we get a generate arrangement $y$ with $\|x-y\|_1\leq \eps\|x\|_1$, which contradicts $x\in\mathcal R_\eps$.
\end{proof}

\begin{proposition}\label{<p>}
%  Take any $\eps\in(0,1)$.
  All $x\in\mathcal R_{\eps}$ satisfy $p^*_x\in(Q_{\eps, N},1-Q_{\eps, N})$ where $Q_{\eps, N}=\eps/(2N^2).$ 
\end{proposition}

\begin{proof}
  Note that $(1-Q_{\eps, N})^N\geq 1-\frac{\eps}{2N}$ due to Bernoulli's inequality.
	
	In view of convexity of scalar product $x\cdot w(p)$ in $p$ (Proposition~\ref{convex}), it suffices to show that $h(x,Q_{\eps, N})\leq 0$ and $h(x,1-Q_{\eps, N})\geq 0$. We prove the first inequality; the proof of the second one is similar. We suppose by contradiction that $h(x,Q_{\eps, N})> 0$ for some $x\in R_{\eps}$. Then
 \begin{multline}\label{h(x,Q)}
    0<h(x,Q_{\eps, N})=\sum_{i>0}i(x_{i,path~1}+x_{i,path~0})Q_{\eps, N}^{i-1}-\sum_{j>0}j(x_{j,path~2}+x_{j,path~0})(1-Q_{\eps, N})^{j-1}\leq\\
    x_{1,path~1}+x_{1,path~0}+\frac{\eps}{2N}\sum_{i>1}i(x_{i,path~1}+x_{i,path~0})-\left(1-\frac{\eps}{2N}\right)\sum_{j>0}j(x_{j,path~2}+x_{j,path~0})
      \leq \\
       x_{1,path~1}+x_{1,path~0}-\sum_{j>0}j(x_{j,path~2}+x_{j,path~0})+\frac{\eps}{2N}\sum_{i>1}i(x_{i,path~1}+x_{i,path~2}+2x_{i,path~0}).
  \end{multline}
Consider vector $y,y_{1,path~1}=x_{1,path~1}+\frac{\eps}{2N}\sum_{i>1}i(x_{i,path~1}+x_{i,path~2}+2x_{i,path~0})$ and $y_{i,path~j}=x_{i,path~j}$ for all other $(i,j)$. By using inequality (\ref{h(x,Q)}), we get that 
$$\sum_{j>0}j(y_{j,path~2}+y_{j,path~0})\leq y_{1,path~1}+y_{1,path~0},$$
that means that $y$ is degenerate. Furthermore, $\|x-y\|_1\leq\eps\|x\|_1$. Consequently, $x$ is $\eps$-degenerate, which contradicts $x\in\mathcal R_\eps$. Hence, $h(x,Q_{\eps, N})\leq 0$ for all $x\in R_{\eps}$ as desired.
\end{proof}

\def\grad{\mathop{\mathrm{grad}}}
\begin{lemma}\label{p_x-p_y}
	For any $\eps>0$ and $N\in\mathbb{N}$, there exists a positive number $P_{\varepsilon, N}$, such that for any vectors $x,y\in\mathcal R_\eps$, we have
	$$
	|p^*_x-p^*_y|\leq \frac{P_{\eps,N}\|x-y\|_1}{\min(\|x\|_1,\|y\|_1)}.
	$$
\end{lemma}

\begin{proof}
	Recall that the function $p^*$ is implicitly defined on $\mathcal R_\eps$ by the equation $h(x,p^*_x)=0$. Therefore, by the Implicit function theorem, for any $(i,j)\in |N|\times\{0,1,2\}$ and $a\in\mathcal \mathcal R_\eps$ we have
	$$
	\frac{\partial p^*_a}{\partial x_{i, path~j}}=-\frac{\frac{\partial h}{\partial x_{i,path~j}}(a,p^*_a)}{\frac{\partial h}{\partial p}(a,p^*_a)};
	$$
	the denominator does not vanish, due to Proposition~\ref{convex}.
	
	It follows from \eqref{h_x} that the absolute value of the numerator does not exceed $N$. On the other hand,
	\begin{equation}
	\label{h'}
	\frac{\partial h}{\partial p}(a,p^*_a)
	=w''(p^*_a)\cdot a\geq \left(p^*_a\right)^{N-2}\|a\|_{1,2},
	\end{equation}
	where we used that $p\leq(1-p).$
	
	Therefore,
	\begin{gather}\label{bb}
	\left|\frac{\partial p^*_a}{\partial x_{i,path~j}}\right|\leq \frac{N\left(p^*_a\right)^{2-N}}{\|a\|_{1,2}}.
	\end{gather}
	
	Since $x,y\in\mathcal R_\eps$, all their coordinates are positive. So, for any vector $a\in[x,y]$, we have $\|a\|_{1,2}\geq\min(\|x\|_{1,2},\|y\|_{1,2})$. Moreover, $a$ is $\eps$-regular, since $\mathcal R_\eps$ is convex; so $p^*_a\geq Q_{\eps, N}$ by Proposition~\ref{<p>}. Therefore, denoting by $\|v\|_2$ the Euclidean norm of a vector $v$, by means of Lagrange's mean value theorem for vector-valued function (for example, see \cite{Matk}), we have
	\begin{multline*}
	|p^*_x-p^*_y|\leq \|x-y\|_2 \cdot\max_{a\in[x,y]}\|\grad p^*_a\|_2\leq \|x-y\|_1 \cdot\max_{a\in[x,y]}\|\grad p^*_a\|_2
	\leq \|x-y\|_1\cdot \max_{a\in[x,y]}\sqrt{3N}\frac{N\left(p^*_a\right)^{2-N}}{\|a\|_{1,2}}\\
	\leq \|x-y\|_1\cdot\frac{\sqrt {3N}\cdot NQ_{\eps, N}^{2-N}}{\min(\|x\|_{1,2},\|y\|_{1,2})}\leq \frac{\sqrt{3}N^{3/2}Q_{\eps, N}^{2-N}}{\eps} \cdot\frac{\|x-y\|_1}{\min(\|x\|_1,\|y\|_1)},
	\end{multline*}
	where in the third inequality we upper bounded Euclidean norm of $\grad p^*_a=\left[\frac{\partial p^*_a}{\partial x_{i,path~j}}\right]_{i,j}$ by using (\ref{bb}), and in the last inequality we used Proposition \ref{reg-good}.
\end{proof}

\begin{lemma}\label{yw(p)-e(y)}
	For any $\eps>0$ and $N\in\mathbb{N}$, there exists a positive $C_{\varepsilon, N}$, such that for any vectors $x,y\in R_\eps$ with $\frac{\|x\|_1}{2}\leq \|y\|_1\leq \|x\|_1$, we have
	$$
	y\cdot w(p^*_x)-e(y)\leq C_{\eps, N}\frac{\|x-y\|^2_1}{\|x\|_1}.
	$$
	\label{Taylor}
\end{lemma}

\begin{proof}
	Since the claim of the lemma is dimension-free (both parts multiply by $\lambda$ under the change $(x,y)\mapsto (\lambda x,\lambda y)$), we may assume that $\|x\|=1$, so $\|y\|\in\bigl[\frac12,1\bigr]$.
	
	Next, we have $h(y,p^*_y)=0$. By using Taylor's formula with Lagrange remainder for function $y\cdot w(\cdot)$ in the point $p^*_y$, we have
\begin{gather*}
	y\cdot w(p^*_x)-e(y)=y\cdot w(p^*_x)-y\cdot w(p^*_y)
	\leq \frac{\bigl(p^*_x-p^*_y\bigr)^2}2 \max_{\text{$q$ between $p^*_x$ and $p^*_y$}}\left|\frac{\partial h}{\partial p}(y,q)\right|.
\end{gather*}
We apply Lemma \ref{p_x-p_y} and use $\bigl|\frac\partial{\partial p}h(a,q)\bigr|\leq N^2\|a\|_1$ for any $q\in[0,1]$ and $a\in\mathcal A$ (refer to the explicit expansion of $\frac{\partial}{\partial p}h(a,p)$).
\begin{gather*}
	\frac{\bigl(p^*_x-p^*_y\bigr)^2}2 \max_{\text{$q$ between $p^*_x$ and $p^*_y$}}\left|\frac{\partial h}{\partial p}(y,q)\right|\leq \frac{P_{\eps, N}^2\|x-y\|^2_1}{2\|y\|^2_1} N^2\|y\|_1\leq N^2P_{\eps, N}^2\frac{\|x-y\|^2_1}{\|x\|_1},
\end{gather*}
as desired.
\end{proof}

Recall that by $r$ we denote a running vector and $r_{path~j}$ denotes the projection of running vector $r$ on $path~j$.

\section{Proof of Proposition 1}
\bigskip
For $v\in\mathbb{R}^{3|N|+3}$ determine a new norm
$$
q(v)=\sum_{i>0} i|v_i|. 
$$

From now on, we fix an arbitrary $\eps\in(0,1)$. 

Formula~\eqref{balance} from Subsection~\ref{Remover's strategy} hints that, in order to (almost) preserve the weight of the arrangement, it is convenient for Pusher to choose the shift vector $r$ satisfying
$$
  r\cdot w(p^*_x) 
  =\frac1{p^*_x}(r_{path~1}+r_{path~0})\cdot w(p^*_x)_{path~1} 
    = \overrightarrow r(\tau_1)\cdot w(p^*_x),
$$
In view of~\eqref{balance}, this yields that
$$
  r\cdot w(p^*_x) 
  =\frac1{1-p^*_x}(r_{path~2}+r_{path~0})\cdot w(p^*_x)_{path~2}
    = \overrightarrow r(\tau_2)\cdot w(p^*_x)
$$
as well. In other words, if Pusher chooses $r$ as the running part at arrangement~$x$, then for every response $\tau$ of Remover we have
$$
  \bigl((x-r)+\overrightarrow r(\tau)\bigr)\cdot w(p^*_x)
    =x\cdot w(p^*_x),
$$
i.e., the $p^*_x$-weight of the arrangement does not change. We say that such vector $r\in\mathcal A$ is \emph{balanced with respect to $x$}. (Surely, we always assume that $r_{0,\path j}=0$ for all $j$.)

\bigskip
The following two  lemmas describes an ``almost optimal'' Pusher's strategy on an $\eps$-regular arrangement. We start with choosing a ``direction'' of the shift, and then we find an appropriate multiple of that direction as an actual shift.

\begin{lemma}
  \label{direction-IB}
  For any $x\in\mathcal R_\eps$, there is a vector $d\in \mathcal A$ balanced with respect to~$x$ such that $x-d\in\mathcal A$ and $q(d)\geq \frac1Nq(x)$.%, where $\delta_{\eps,N}$ is a positive constant depending only on~$\eps$ and~$N$.
\end{lemma}

\begin{proof}
  Recall that, by the definition of $p^*_x$, we have $\frac\partial{\partial p} w(p)\big|_{p=p^*_x}\cdot x=0$, or
  $$
    0=\frac\partial{\partial p} w(p)\bigg|_{p=p^*_x}\cdot x=\sum_{k>0}k\bigl((x_{k,path~1}+x_{k,path~0})(p^*_x)^{k-1}
      -(x_{k,path~2}+x_{k,path~0})(1-p^*_x)^{k-1}\bigr).
  $$
  This means that
  \begin{equation}
    \label{r-balanced}
    \sum_{k>0}k(x_{k,path~1}+x_{k,path~0})(p^*_x)^{k-1}
    =\sum_{k>0}k(x_{k,path~2}+x_{k,path~0})(1-p^*_x)^{k-1}.
  \end{equation}
%  where 
%  $$
%    A_k=(x_{k,path~1}+x_{k,path~0})(p^*_x)^{k-1}
%      -(x_{k,path~2}+x_{k,path~0})(1-p^*_x)^{k-1}
%  $$
  
  Now we define the vector $d$ by
  $$
    d_{k,\path j}=\frac kN x_{k,\path j}.
  $$
  Clearly, both $d$ and $x-d$ lie in~$\mathcal A$, and $q(d)\geq \frac1Nq(x)$. Finally, formula~\eqref{r-balanced}  reads $\overrightarrow d(\tau_1)\cdot w(p^*_x)=\overrightarrow d(\tau_2)\cdot w(p_x^*)$, which by~\eqref{balance} means that $d$ is balanced with respect to~$x$, as desired.
\end{proof}

\begin{lemma}
  \label{move}
  For any $x\in\mathcal R_\eps$, Pusher can perform a move such that, after an arbitrary response of Remover, the resulting arrangement $y$ satisfies the following conditions:
  \begin{gather}
     e(x)-e(y)\leq \eps(q(x)-q(y)),
     \label{eps}\\
    q(y)\leq (1-\delta_{\eps,N})q(x), 
    \label{delta}
  \end{gather}
  where $\delta_{\eps,N}$ is a constant depending only on $\eps$ and $N$.
\end{lemma}

\begin{proof}
  Choose a vector $d\in\mathcal A$ satisfying the requirements in Lemma~\ref{direction-IB}. We define the actual vector of Pusher's shift as
  $$
    r=\mu d, \qquad\text{where}\quad
    \mu=\min\left\{\frac\eps4,\frac\eps{4C_{\eps/2,N}}\right\},
  $$
  where the constant $C_{\eps/2,N}$ is taken from Lemma~\ref{yw(p)-e(y)}. Notice here that $\|r\|_1\leq\frac\eps4\|d\|_1\leq\frac\eps4\|x\|_1$.

  \medskip
  Denote by $y_i=(x-r)+\overrightarrow r(\tau_i)$, $i=1,2$, the two possible configurations after a Remover's response. Observe that $\|y_i\|_1\leq\|x\|_1$ and
  \begin{equation}
    \label{x-y<}
    \|x-y_i\|_1\leq\|r\|_1+\|\overrightarrow r(\tau_i)\|_1
    \leq 2\|r\|_1
    \leq 2\mu\|x\|_1
    \leq \frac\eps2\|x\|_1;
  \end{equation}
  in particular, this yields $\|y_i\|_1\geq\frac{\|x\|_1}2$.
    
  Suppose, for the sake of contradiction, that some $y_i$ is $\eps/2$-degenerate. so that there exists a $z\in\mathcal D$ such that $\|y-z\|_1\leq \frac\eps2\|y\|_1$. Then
  $$
    \|x-z\|_1\leq\|x-y\|_1+\|y-z\|_1
    \leq \frac\eps2\|x\|_1+\frac\eps2\|x\|_1
    =\eps\|x\|_1,
  $$
  so $x$ was $\eps$-degenerate. This contradicts the assumptions of the Lemma. Hence $y_1,y_2\in\mathcal R_{\eps/2}$. 
  
  \bigskip
  Set $y=y_i$ and $\tau=\tau_i$ for an arbitrary $i\in\{1,2\}$. Our next aim is to establish~\eqref{delta}.
  
  Indeed, we have
  \begin{equation}
    \label{q-q}
    q(x)-q(y)=q(r)-q(\overrightarrow r(\tau))
    \geq \sum_{j=0}^2 \sum_{i>0}ir_{i,\path j}
    -\sum_{j=0}^2\sum_{i>0} (i-1)r_{i,\path j}
    \geq \|r\|_1.
  \end{equation}
  Since $\|d\|_1\geq \frac1N q(d)\geq\frac1{N^2}q(x)$, we obtain
  $$
    q(x)-q(y)\geq \mu\|d\|_1\geq\frac\mu{N^2}q(x),
  $$
  which shows that \eqref{delta} holds with
  $$
    \delta_{\eps,N}=\frac{\mu}{N^2}.
  $$
  Notice here that~\eqref{q-q} yields also that
  \begin{equation}
    q(x)-q(y)\geq\|r\|_1\geq\frac{\|x-y\|_1}2,
    \label{q-q<x-y}
  \end{equation}
  due to~\eqref{x-y<}.

  \bigskip
  Finally, we are about to prove~\eqref{eps}.
  Since $d$ is balanced with respect to~$x$, we have $y\cdot w(p^*_x)=x\cdot w(p^*_x)$, so that
  $$
    e(x)-e(y)=x\cdot w(p_x^*)-e(y)=y\cdot w(p_x^*)-e(y).
  $$
  Since $\frac{\|x\|_1}2\leq\|y\|_1\leq\|x\|_1$, we can apply Lemma~\ref{yw(p)-e(y)} to estimate the right-hand part as
  $$
    y\cdot w(p_x^*)-e(y)\leq C_{\varepsilon/2,N}\frac{\|x-y\|^2_1}{\|x\|_1}
    =C_{\varepsilon/2,N}\frac{\|x-y\|_1}{\|x\|_1}
      \|x-y\|_1.
  $$
  Recall that $\mu C_{\eps/2,N}\leq \frac\eps4$ and $\|x-y\|_1\leq 2\mu\|x\|_1$ by~\eqref{x-y<}; hence the above inequality extends as
  $$
    y\cdot w(p_x^*)-e(y)
    \leq C_{\eps/2,N}\cdot 2\mu\cdot \|x-y\|_1
    \leq\frac\eps2\|x-y\|_1\leq \eps(q(x)-q(y)),
  $$
  where the last inequality holds by~\eqref{q-q<x-y}. This proves~\eqref{eps}.
\end{proof}

\begin{proposition}
  For any $\eps\in(0,1)$ and any $x\in\mathcal A$, the inequality ~$e(x)-E(x)\leq 2N\varepsilon\|x\|_1$ holds. Thus, $e(x)=E(x)$.
\end{proposition}

\begin{proof}
  Let Pusher perform moves described in Lemma~\ref{move} while the appearing arrangements are $\eps$-regular. let $x=x^0,x^1,\dots$ denote the sequence of arrangements appearing before Pusher's moves. Due to \eqref{eps}, we have $e(x^{i})-\varepsilon\cdot q(x^{i})\leq e(x^{i-1})-\varepsilon\cdot q(x^{i-1})$ for all $i$. If this process lasts indefinitely, the value of $q(x^i)$ becomes arbitrarily small, due to~\eqref{delta}. So, eventually one of the following two options occurs.

  \smallskip
  \noindent
  \textit{Case 1: Some arrangement $x^s$ satisfies $q(x^s)\leq \eps\cdot q(x)$.} This means that $E(x^s)\geq x^s_{0,path~0}+x^s_{0,path~1}+x^s_{0,path~2}\geq \|x^s\|_1-q(x^s)\geq \|x^s\|_1-\eps\cdot q(x)$. 
  Therefore, we have
  $$
    e(x)-\eps\cdot q(x)\leq e(x^s)-\eps\cdot q(x^s)\leq \|x^s\|_1\leq E(x^s)+\eps\cdot q(x),
  $$
  so $E(x)\geq E(x^s)\geq e(x)-2\eps\cdot q(x)\geq e(x)-2\varepsilon N\|x\|_1$.
  
  \smallskip
  \noindent
  \textit{Case 2. Some arrangement $x^s$ is $\eps$-degenerate (so Pusher cannot proceed on).} By Proposition~\ref{e-deg}, we have $E(x^s)\geq e(x^s)-2\eps\|x^s\|_1$. Therefore,
  $$
    e(x)-\eps\cdot q(x)\leq e(x^s)-\eps\cdot q(x^s)\leq E(x^s)+2\eps\|x^s\|_1\leq E(x)+2\eps\|x\|_1,
  $$
  so $E(x)\geq e(x)-2\eps\|x\|_1-\eps\cdot q(x)\geq e(x)-\varepsilon(N+2)\|x\|_1.$
\end{proof}

\renewcommand{\refname}{References}

\end{document}